\documentclass{article}
\usepackage{amsmath,amsthm,amssymb}
\usepackage{latexsym}

\newcommand{\FF}{\mathbb{F}}

\theoremstyle{theorem}
\newtheorem{theorem}{Theorem}[section]

\newtheorem{lemma}[theorem]{Lemma}
\newtheorem{proposition}[theorem]{Proposition}

\theoremstyle{defi}

\begin{document}

\title{\textbf{THE UNIT GROUP OF SMALL GROUP ALGEBRAS AND THE MINIMUM COUNTEREXAMPLE TO THE ISOMORPHISM PROBLEM}}

\author{Leo Creedon\\
Institute of Technology Sligo\\
Sligo, IRELAND\\
 e-mail: creedon.leo@itsligo.ie}
\date{}

\maketitle

\begin{abstract}
\noindent Let $KG$ denote the group algebra of the group $G$ over the field $K$ and let $U(KG)$ denote its group of units.
Here without the use of a computer we give presentations for the unit groups of all group algebras $KG$, where the size of $KG$ is less than $1024$.  
As a consequence we find the minimum counterexample to the Isomorphism Problem for group algebras.
\end{abstract}

\section{Introduction}

\noindent Let $KG$ denote the group algebra of the group $G$ over the
field $K$. Let $U(KG)$ denote its group of invertible elements. For
further details and background see \cite{Passman}.

In \cite{Sandling} presentations are given for $U({\FF}_2 G)$, where ${\FF}_2$ if the Galois field with 2 elements and 
$G$ is a group of order dividing 16.  Using this and a result of Higman \cite{Higman} and new results we provide a 
list of presentations of unit groups $U(KG)$ of all group algebras $KG$ with $|KG|<1024$.  This is done using algebraic techniques and without the use of a computer.
The motivation for seeking these presentations is the same as Sandling's \cite{Sandling}: 
to compile data for use in formulating conjectures and testing hypotheses.  
For example a look at the data here shows that
the Isomorphism Problem for group algebras is false and gives a minimum counterexample. 
In particular ${\FF}_5 C_4$ is isomorphic to ${\FF}_5 C_2 \times C_2$.  
These algebras of order 625 are the minimal group algebras with the property that $KG \cong KH$ with 
$G$ and $H$ not isomorphic.

\vspace{-3mm}

\subsection{Background}

Let $C_n$ denote the cyclic group with $n$ elements, let $D_n$ denote the dihedral group with $n$ elements and 
let $Q_8$ denote the quaternion group with $8$ elements. ${\FF}_{p^n}$ will denote the finite field with $p^n$ elements, 
where $p$ is a prime number. If $R$ and $S$ are rings then $R \oplus S$ denotes the direct sum of $R$ and $S$.  
If $G$ and $H$ are groups then $G \times H$ denotes the direct product of $G$ and $H$. 
If $R$ is a ring then $R^n$ denotes the direct sum of $n$ copies of $R$.  
If $G$ is a group then $G^n$ denotes the direct product of $n$ copies of $G$. If $\alpha$ is an element of a group ring $RG$
then let $aug(\alpha) \in R$ denote the augmentation of $\alpha$. 
$|R|$ is the size of $R$ and $exp (G)$ is the exponent of the group $G$.

The following result will be used throughout:

\begin{lemma}
$U({\FF}_{p^k} C_p^n) = C_p ^{kp^n - k} \times C_{p^k -1}$, where $p$ is any prime number.
\end{lemma}

\begin{proof}
Let $\alpha = a_1 x_1 + a_2 x_2 + \ldots + a_{p^n} x_{p^n} \in {\FF}_{p^k} C_p^n$, 
where the elements $a_i$ are in the field ${\FF}_{p^k}$ and $\{ x_i|i=1,\ldots ,p^n\} $ is a listing of the elements of $C_p^n$.  

Then $\alpha ^p = a_1^p x_1^p + a_2^p x_2^p + \ldots + a_{p^n}^p x_{p^n}^p = (a_1 + a_2 + \ldots + a_{p^n})^p = aug(\alpha)^p$. 
Thus $|U({\FF}_{p^k} C_p^n)| = (p^k)^{p^n-1} (p^k -1) = (p^{kp^n-k})(p^k -1)$. 
Note that $U({\FF}_{p^k} C_p^n)$ has exponent $p(p^k -1)$.  Thus $U({\FF}_{p^k} C_p^n) = C_p ^{kp^n - k} \times C_{p^k -1}$.

\end{proof}

The following result appears in \cite{Higman}:

\begin{proposition}
If $KG$ is a semisimple group algebra and $K$ contains a primitive $m^{th}$-root of unity, 
where $m = exp (G)$ and $n = |G|$ then $KG \cong K^n$.
\end{proposition}

\section{The Unit Groups}

\bigskip

The Table gives the presentations of all unit groups of group algebras of order less that $1024$.  The cases where $p>3$ and $|G|<4$ are easily covered by the following remarks.  For $p>3, \; F_{p^n}C_1 \cong \FF_{p^n}, \; \FF_{p^n}C_2\cong \FF_{p^n}^2$. $\FF_{p^n}C_3 \cong \FF_{p^n}^3$ if $3|p^n-1$ by \cite{Higman}.  If $3\nmid p^n-1$ then $\FF_{p^n}C_3 \cong \FF_{p^n} \oplus \FF_{p^{2n}}$.
The first column of the Table gives the group algebra $KG$ and the second column gives the order $KG$ of the group algebra.
The third column gives the decomposition of the group algebra as the direct sum of matrix rings over fields
(in the semisimple cases) or as the direct sum of group algebras where useful in the non-semisimple cases.
Column four gives the size of the unit group $|U(KG)|$ and column five gives the structure of the unit group $U(KG)$, 
either by giving the presentation or as a direct sum of known groups.  
Table 1 is arranged into sections, with one section being used for each of the different fields $K$ used. Proofs for some of the commutative cases are contained in \cite{Danchev} and \cite{JoeThesis} and alternative proofs are included here for completeness.

\begin{proof}

${\FF}_{2^k} C_1 \cong {\FF}_{2^k} $ and has unit group $C_{2^k-1}$.
$U({\FF}_{2^k} C_2) \cong C_2 ^k \times  C_{2^k-1}$ by the Lemma.
For ${\FF}_{2^k} C_3$, if $3|2^k -1$ then by \cite{Higman}, ${\FF}_{2^k} C_3 \cong {\FF}_{2^k}^3$ and has unit group $C_{2^k-1}^3$. If $3\nmid 2^k -1$ then ${\FF}_{2^k} C_3 \cong {\FF}_{2^k} \oplus {\FF}_{2^{2k}}$ (since the unit group has an element of order 3) and $U({\FF}_{2^k} C_3) \cong C_{2^k-1} \times C_{2^{2k}-1}$.

$U({\FF}_2 C_1)$ is trivial. $U({\FF}_2 C_2) = C_2$ by the Lemma or \cite{Sandling}.  
Note that for a semisimple group algebra $KG$ with $K$ and $G$ finite, we have that $K$ appears at least once as a 
summand in the Artin-Wedderburn decomposition.
For ${\FF}_2 C_3$, Maschke's Theorem applies and the unit group contains an element of order $3$, 
so ${\FF}_2 C_3 \cong {\FF}_2 \oplus {\FF}_{2^2}$.
$U({\FF}_2 C_2 \times C_2) = C_2 ^3$ by the Lemma or \cite{Sandling}.
For ${\FF}_2 C_4$, note that if $\alpha \in U({\FF}_2 C_4)$ then $\alpha^4 = aug (\alpha) \in {\FF}_2$. 
Thus $\alpha \in U({\FF}_2 C_4)$ iff $aug (\alpha) =1$.  So $|U({\FF}_2 C_4)| = 8$. 
But $\alpha^4=1$ for all $\alpha \in U({\FF}_2 C_4)$ (and there is an element of order $4$, 
so $U({\FF}_2 C_4)$ is an abelian group of exponent $4$ and order $8$, so $U({\FF}_2 C_8) \cong C_2 \times C_4$.

For ${\FF}_2 C_5$, Maschke's Theorem applies, and ${\FF}_2$ appears as a summand.  But the unit group contains an element of
order $5$, so ${\FF}_{2^4}$ also appears.  Thus ${\FF}_2 C_5 \cong {\FF}_2 \oplus {\FF}_{2^4}$.
${\FF}_2 C_6 \cong {\FF}_2 C_3 \times C_2$. This is isomorphic to the group ring 
$({\FF}_2C_3)C_2 \cong ({\FF}_2 \oplus {\FF}_{2^2})C_2 \cong ({\FF}_2 C_2) \oplus ({\FF}_{2^2} C_2)$.
Thus $U({\FF}_2 C_6) \cong U({\FF}_2 C_2) \times U({\FF}_{2^2} C_2) \cong C_2 \times (C_2^2 \times C_3)$ by the Lemma.
Thus $U({\FF}_2 C_6) \cong C_2^3 \times C_3$.

For ${\FF}_2 D_6$, the unit group is calculated in \cite{Indians} as $D_{12}$ with the following presentation. 
If $D_6 = \langle x,y|x^3=y^1=1, yxy=x^2 \rangle $ and $w=1+x^2+y+xy+x^2y$ then $U({\FF}_2 D_6) \cong \langle w,y|w^6=y^2=1, ywy=w^5 \rangle $.

${\FF}_2 C_7$ is semisimple and ${\FF}_2$ appears as a direct summand in the decomposition.  But ${\FF}_2 C_7$ contains an
element of order $7$, so ${\FF}_{2^3}$ or ${\FF}_{2^6}$ appear as direct summands.  But for any $\alpha \in {\FF}_2 C_7$, 
$\alpha^8 = \alpha$, so the unit group has exponent 7.  Thus ${\FF}_2 C_7 \cong {\FF}_2 \oplus {\FF}_{2^3}^2 $ or
${\FF}_2 C_7 \cong {\FF}_2^4 \oplus {\FF}_{2^3} $.  But ${\FF}_2 C_7$ does not have just trivial units \cite{Danchev}, so
${\FF}_2 C_7 \cong {\FF}_2 \oplus {\FF}_{2^3}^2 $ and $U({\FF}_2 C_7) \cong C_7^2$.

$U({\FF}_2 C_2^3)\cong C_2^7$ by the Lemma or \cite{Sandling}.

For ${\FF}_2 C_2 \times C_4$, note that for all $\alpha \in {\FF}_2 C_2 \times C_4$, $\alpha^4 = aug (\alpha) \in {\FF}_2 $.
Thus $exp (U({\FF}_2 C_2 \times C_4)) =4$. Note also that $|U({\FF}_2 C_2 \times C_4)|=2^7$.  
Thus $U({\FF}_2 C_2 \times C_4) = C_2^5 \times C_4$ (with $2^6$ elements of order 2) or $C_2^3 \times C_4^2$ (with $2^5$ elements of order 2) or $C_2 \times C_4^3$ (with $2^4$ elements of order 2).
Next we count the number of elements of order $2$ in $U({\FF}_2 C_2 \times C_4)$.  
Let $C_2 \times C_4 = \langle x,y|x^4=y^2=1=[x,y]\rangle$.
Let $\alpha = a_0+a_1x+a_2x^2+a_3x^3+a_4y+a_5xy+a_6x^2y+a_7x^3y$. Then $\alpha^2=a_0+a_1x^2+a_2+a_3x^2+a_4+a_5x^2+a_6+a_7x^2 = (a_0+a_2+a_4+a_6)+(a_1+a_3+a_5+a_7)x^2 = 1 \Leftrightarrow (a_0+a_2+a_4+a_6)=1$ and $(a_1+a_3+a_5+a_7)=0$. This gives $8^2$ choices for the $a_i$, so there are $2^6$ elements of order 1 or 2. Thus $U({\FF}_2 C_2 \times C_4) = C_2^5 \times C_4$.

For ${\FF}_2 C_8$ note that if $\alpha \in {\FF}_2 C_8$ then $\alpha^8 = aug(\alpha) \in {\FF}_2$.  
Thus $\alpha$ is a unit iff $aug(\alpha)=1$ so $|U({\FF}_2 C_8)|=2^7$. Note that $U({\FF}_2 C_8)$ has exponent 8.
Thus $U({\FF}_2 C_8) \cong C_8 \times C_8 \times C_2$ (with $2^3$ elements of order 2) or $C_8\times C_4 \times C_4$
(also with $2^3$ elements of order 2) or $C_8\times C_4\times C_2 \times C_2$ (with $2^4$ elements of order 2) or $C_8\times C_2^4$ (with $2^5$ elements of order 2).  Next we count the number of elements of order 2 in $U({\FF}_2 C_8)$.
Let $ \alpha = \Sigma_{i=0}^7 a_ix^i \in U({\FF}_2 C_8)$ with $\alpha^2=1$.  Then $\alpha^2=\Sigma_{i=0}^7 a_i x^{2i} =
(a_0+a_4)1+(a_1+a_5)x^2+(a_2+a_6)x^4+(a_3+a_7)x^6=1$. This gives $2^4$ choices.  Hence $U({\FF}_2 C_8) \cong C_8\times C_4\times C_2 \times C_2$.

For ${\FF}_2 D_8$, let $D_8 = \langle x,y|x^4=y^2=1, [x,y]=x^2 \rangle$ and let $a = x+y+xy \in {\FF}_2 D_8$.  
Then $U({\FF}_2 D_8)$ is given in \cite{Sandling} as $U({\FF}_2 D_8) \cong 
\langle x,y,a| x^4=y^2=[a,x]^2=[a,y]^2=a^4=1, [x,y]=x^2, [a^2,x]=[a^2,y]=[a,x,y]=[x^2,a]=1 \rangle $.

For ${\FF}_2 Q_8$, let $Q_8 = \langle x,y|x^4=1, x^2=y^2, [x,y]=x^2 \rangle$ and let $a = x+y+xy \in {\FF}_2 Q_8$.  
Then $U({\FF}_2 Q_8)$ is given in \cite{Sandling} as $U({\FF}_2 Q_8) \cong 
\langle x,y,a| x^4=[a,x]^2=[a,y]^2=a^4=1, x^2=y^2, [x,y]=x^2, [a^2,x]=[a^2,y]=[a,x,y]=[x^2,a]=1 \rangle $.

${\FF}_2 C_3 \times C_3$ is isomorphic to the group ring $({\FF}_2 C_3)C_3$.  By a previous result ${\FF}_2 C_3 \cong {\FF}_2 \oplus {\FF}_{2^2}$. Thus  ${\FF}_2 C_3 \times C_3 \cong ({\FF}_2 \oplus {\FF}_{2^2})C_3 \cong {\FF}_2 C_3 \oplus {\FF}_{2^2}C_3 $.  By \cite{Higman} ${\FF}_{2^2}C_3 \cong {\FF}_{2^2}^3$.  Thus ${\FF}_2 C_3 \times C_3 \cong {\FF}_2 \oplus {\FF}_{2^2}^4$ and $U({\FF}_2 C_3 \times C_3) \cong C_3^4$.

For ${\FF}_2 C_9$, note that Maschke's Theorem applies and the unit group contains an element of order 9, so 
${\FF}_2 C_9 \cong {\FF}_2 \oplus {\FF}_{2^6} \oplus$ other summands. Also $\hat x, \hat {x^3} +\hat x$ and $x^3+x^6$ are orthogonal central idempotents, so ${\FF}_2 C_9 = {\FF}_2 C_9 \hat x \oplus {\FF}_2 C_9 (\hat {x^3} + \hat x) \oplus {\FF}_2 C_9 (x^3+x^6)$.  But ${\FF}_2 C_9 \hat x \cong {\FF}_2$ and ${\FF}_2 C_9 (\hat {x^3}+\hat x) \cong {\FF}_{2^2}$, so we must have ${\FF}_2 C_9 \cong {\FF}_2 \oplus {\FF}_{2^2} \oplus {\FF}_{2^6}$.

Next we consider group algebras where the coefficient field is ${\FF}_{2^2}$.
${\FF}_{2^2}C_1 \cong {\FF}_{2^2}$ and has unit group $C_3$.
By the Lemma, $U({\FF}_{2^2}C_2) \cong C_2^2 \times C_3$.
By \cite{Higman} ${\FF}_{2^2}C_3 \cong {\FF}_{2^2}^3$.

By the Lemma, $U({\FF}_{2^2}C_2 \times C_2) \cong C_2^6 \times C_3$.

For $U({\FF}_{2^2}C_4)$, note that if $\alpha \in {\FF}_{2^2}C_4$ then $\alpha^4=aug(\alpha) \in {\FF}_{2^2}$.
So the exponent of $U({\FF}_{2^2}C_4)$ is 12 and $|U({\FF}_{2^2}C_4)| = 4^3.3 = 2^6.3$. 
Thus $U({\FF}_{2^2}C_4) = C_2^4\times C_4\times C_3$ (with $2^5$ elements of order 2) or $C_2^2\times C_4^2 \times C_3$ (with $2^4$ elements of order 2) or $C_4^3\times C_3$ (with $2^3$ elements of order 2). Next we count the number of elements of order 2 in $U({\FF}_{2^2}C_4)$.  If $\alpha = \Sigma_{i=0}^3 a_i x^i \in {\FF}_{2^2}C_4$ then $\alpha^2=(a_0^2+a_2^2)+(a_1^2+a_3^2)x^2=1 \Leftrightarrow a_0^2+a_2^2=1$ and $a_1^2+a_3^2=0$, so these $2^2.2^2$ choices give $2^4$ elements of order 2. Thus $U({\FF}_{2^2}C_4) = C_2 ^2 \times C_4 ^2 \times C_3$.

Next we consider group algebras where the characteristic of the coefficient field is 3.
${\FF}_{3^k} C_1 \cong {\FF}_{3^k}$ so its unit group is $C_{3^k -1}$.
${\FF}_{3^k} C_2 \cong {\FF}_{3^k} \oplus {\FF}_{3^k}$ so its unit group is $C_{3^k -1}^2$.
$U({\FF}_{3^k} C_3) \cong C_3 ^{2k} \times C_{3^k -1}$ by the Lemma.

${\FF}_3 C_2 \times C_2 \cong {\FF}_3^4$ by \cite{Higman}.
For ${\FF}_3 C_4$ note that Maschke's Theorem applies and there is a unit of order 4, so ${\FF}_3 C_4 \cong {\FF}_3 \oplus {\FF}_{3^2} \oplus$ other summands.  Thus ${\FF}_3 C_4 \cong {\FF}_3^2 \oplus {\FF}_{3^2}$.
For ${\FF}_3 C_5$ note that Maschke's Theorem applies and there is a unit of order 5, so ${\FF}_3 C_5 \cong {\FF}_3 \oplus {\FF}_{3^4}$.
${\FF}_3 C_6 \cong {\FF}_3 C_2 \times C_3$ is isomorphic to the group ring $({\FF}_3 C_2) C_3 \cong ({\FF}_3\oplus {\FF}_3) C_3 \cong {\FF}_3 C_3 \oplus {\FF}_3 C_3$. So by a previous result it has unit group $C_3^4 \times C_2^2$.

$U({\FF}_3 D_6)$ is given in \cite{Indians} as 
$\langle v_1,v_2,v_3 | v_1^6=v_2^6=v_3^3=[v_1^3,v_2]=[v_1^3,v_3]=[v_2^2,v_1]=[v_2^2,v_3]=1, \; v_3 v_2 = v_1 v_2 v_1 v_3^2, \; v_3v_1=v_2v_1^5v_2^5v_3, \; v_2v_1=v_1^2v_2v_1^2v_2v_1v_2^{-1}v_1^2 \rangle$, where
$D_6=\langle x,y|x^3=y^2=1, \; [x,y]=x \rangle$ and $v_1=-x^2, \; v_2=1-x^2+y, \; v_3=1+(x-x^2)(1-y)$.
For an alternative approach to ${\FF}_{3^k} D_6$ see \cite{LandJ1}.

Next we consider group algebras where the characteristic of the coefficient field is $p>3$.
${\FF}_{p^n} C_1 \cong {\FF}_{p^n}$ so it has unit group $C_{p^n -1}$.
${\FF}_{p^n} C_2 \cong {\FF}_{p^n} \oplus {\FF}_{p^n}$ so it has unit group $C_{p^n -1}^2$.
For ${\FF}_{p^n} C_3$, if $3 | p^n -1$ then by the Proposition, ${\FF}_{p^n} C_3 \cong {\FF}_{p^n}^3$ and has unit group $C_{p^n -1}^3$. If $3 \nmid | p^n -1$ then clearly ${\FF}_{p^n} C3 \cong {\FF}_{p^n} \oplus {\FF}_{p^{2n}}$.
${\FF}_{p^n} C_2 \times C_2 \cong {\FF}_{p^n}^4$ by the Proposition.
For ${\FF}_{p^n} C_4$, if $4 | p^n -1$ then by the Proposition, ${\FF}_{p^n} C_4 \cong {\FF}_{p^n}^4$.
If $4 \nmid | p^n -1$ then ${\FF}_{p^n} C_4 \cong {\FF}_{p^n}^2 \oplus {\FF}_{p^{2n}}$ or ${\FF}_{p^n} \oplus {\FF}_{p^{3n}}$.  But $4 \nmid | p^{3n-1}$, giving ${\FF}_{p^n} C_4 \cong {\FF}_{p^n}^2 \oplus {\FF}_{p^{2n}}$.

\end{proof}
\begin{table}[p]
\begin{center}
    \begin{tabular}{ | c | c | c | c |c| }
          \hline
        $KG$ & $|KG|$ & Direct Sum & $|U(KG)|$ & $U(KG)$ \\ \hline \hline
        ${\FF}_2C_1$ & 2 & ${\FF}_2$ & 1 & $C_1$ \\ \hline
        ${\FF}_2C_2$ & 4 &  & 2 & $C_2$ \\ \hline
        ${\FF}_2C_3$ & 8 & ${\FF}_2\oplus {\FF}_{2^2}$ & 3 & $C_3$ \\ \hline
        ${\FF}_2C_2 \times C_2$ & 16 & $({\FF}_2 C_2)C_2$ & 8 & $C_2^3$ \\ \hline
        ${\FF}_2C_4$ & 16 & & 8 & $C_2 \times C_4$ \\ \hline
        ${\FF}_2C_5$ & 32 & ${\FF}_2 \oplus {\FF_{2^2}} $ & 15 & $C_{15}$ \\ \hline
        $\FF_2C_6$ & 64 & $\FF_2 C_2 \oplus \FF_{2^2}C_2$ & 24 & $C_2^3\times C_3$ \\ \hline
        $\FF_2D_6$ & 64 &  & 12 & $D_{12}$ \\ \hline
        $\FF_2C_7$ & 128 & $\FF_2 \oplus \FF_{2^3}^2$ & 49 & $C_7^2$ \\ \hline
        $\FF_2C_2^3$ & 256 &  & 128 & $C_2^7$ \\ \hline
        $\FF_2C_2\times C_4$ & 256 &  & 128 & $C_2^5 \times C_4$ \\ \hline
        $\FF_2C_8$ & 256 &  & 128 & $C_8\times C_4\times C_2^2$ \\ \hline
                $\FF_2D_8$ & 256 &  & 128 & $\langle x,y,a | x^4=y^2=[a,x]^2=[a,y]^2=a^4=1,$\\
                           &     &  &     & $[x,y]=x^2, [a^2,x]=[a^2,y]=[a,x,y]=[x^2,a]=1 \rangle $\\ 
                           &     &  &     & where $D_8=\langle x,y\rangle $ and a=x+y+xy \\ \hline
                $\FF_2Q_8$ & 256 &  & 128 & $\langle x,y,a | x^4=[a,x]^2=[a,y]^2=a^4=1,y^2=x^2,$\\
                           &     &  &     & $[x,y]=x^2, [a^2,x]=[a^2,y]=[a,x,y]=[x^2,a]=1 \rangle $\\ 
                           &     &  &     & where $Q_8=\langle x,y\rangle $ and a=x+y+xy \\ \hline             
        $\FF_2C_3^2$ & 512 & $\FF_2 \oplus \FF_{2^2}^4$ & 81 & $C_3^4$ \\ \hline
        $\FF_2C_9$ & 512 & $\FF_2 \oplus \FF_{2^2} \oplus F_{2^6}$ & 189 & $C_3\times C_{63}$ \\ \hline \hline
        ${\FF}_{2^2}C_1$ & 4 & ${\FF}_{2^2}$ & 3 & $C_3$ \\ \hline
        ${\FF}_{2^2}C_2$ & 16 &  & 12 & $C_2^2 \times C_3$ \\ \hline
        ${\FF}_{2^2}C_3$ & 64 & ${\FF}_{2^2}^3$ & 27 & $C_3^3$ \\ \hline
        ${\FF}_{2^2}C_2^2$ & 256 & & 192 & $C_2^6 \times C_3$ \\ \hline
        ${\FF}_{2^2}C_4$ & 256 & & 192 & $C_2^2\times C_4^2 \times C_3$ \\ \hline
   			\hline
   			${\FF}_3C_1$ & 3 & ${\FF}_3$ & 2 & $C_2$ \\ \hline
   			${\FF}_3C_2$ & 9 & ${\FF}_3^2$ & 4 & $C_2^2$ \\ \hline
   			${\FF}_3C_3$ & 27 &  & 18 & $C_3^2 \times C_2$ \\ \hline
   			${\FF}_3C_2\times C_2$ & 81 & ${\FF}_3^4$ & 16 & $C_2^4$ \\ \hline
   			${\FF}_3C_4$ & 81 & ${\FF}_3^2 \oplus \FF_{3^2}$ & 32 & $C_2^2 \times C_8$ \\ \hline
   			${\FF}_3C_5$ & 243 & ${\FF}_3\oplus \FF_{3^4}$ & 160 & $C_2\times C_{80}$ \\ \hline
   			${\FF}_3C_6$ & 729 &  & 324 & $C_3^4 \times C_2^2$ \\ \hline
   						${\FF}_3D_6$ & 729 &  & 324 & $\langle v_1,v_2,v_3 | v_1^6=v_2^6=v_3^3=[v_1^3,v_2]=[v_1^3,v_3]=$\\ 
   						           &     &  &     &  $[v_2^2,v_1]=[v_2^2,v_3]=1, v_3 v_2 = v_1 v_2 v_1 v_3^2,$ \\
   						           &     &  &     &  $v_3v_1=v_2v_1^5v_2^5v_3, \; v_2v_1=v_1^2v_2v_1^2v_2v_1v_2^{-1}v_1^2 \rangle$\\
	                      &     &  &     & where $D_6=\langle x,y|x^3=y^2=1, \; [x,y]=x \rangle$ and \\
                         &     &  &     &  $v_1=-x^2, \; v_2=1-x^2+y, \; v_3=1+(x-x^2)(1-y)$ \\ \hline
   			\hline
   			${\FF}_{3^2}C_1$ & 9 & $\FF_{3^2}$ & 8 & $C_8$ \\ \hline
   			${\FF}_{3^2}C_2$ & 81 & $\FF_{3^2}^2$ & 64 & $C_8^2$ \\ \hline
   			${\FF}_{3^2}C_3$ & 729 &  & 648 & $C_3^4\times C_8$ \\ \hline
   		  \hline
   			${\FF}_5C_1$ & 5 & $\FF_5$ & 4 & $C_4$ \\ \hline
   			${\FF}_5C_2$ & 25 & $\FF_5^2$ & 16 & $C_4^2$ \\ \hline
   			${\FF}_5C_3$ & 125 & $\FF_5\oplus \FF_{5^2}$ & 96 & $C_4\times C_{24}$ \\ \hline
   			${\FF}_5C_2\times C_2$ & 625 & $\FF_5^4$ & 256 & $C_4^4$ \\ \hline
   			${\FF}_5C_4$ & 625 & $\FF_5^4$ & 256 & $C_4^4$ \\ \hline		
\end{tabular}
\end{center}

\end{table}

\pagebreak

\end{document}